\documentclass{amsart}
\usepackage{amssymb,amsmath,latexsym,amsthm}
\usepackage[english]{babel}

\usepackage{amsmath,amssymb,amsbsy,amsfonts,amsthm,latexsym,
	amsopn,amstext,amsxtra,euscript,amscd}
\usepackage{hyperref}
\usepackage{array}
\usepackage{ifthen}
\usepackage{url}

\newtheorem{thm}{Theorem}

\newtheorem{lem}[thm]{Lemma}

\begin{document}
	
	\title[]{Linear Combinations of Factorial and $S$-unit in a Ternary recurrence sequence with a double root}
	
		\author[Florian Luca]{Florian Luca}
		\address{School of Mathematics,
			Wits University, South Africa and Research Group in Algebraic Structures, King Abdulaziz University, jeddah, saudi Arabia and Max Planck Institute for Mathematics, Bonn, Germany.}
		
		\email{florian.luca@wits.ac.za.}
		
	\author[Armand Noubissie]{Armand Noubissie}
	\address{
{Graz University of Technology}\newline
{Institute for Analysis and Number Theory}\newline
{M\"{u}nzgrabenstrasse 36/II, 8010 Graz, Austria}}
\email{\tt armand.noubissie@tugraz.at}	
	
	
	\date{\today}

	\maketitle

	\begin{abstract}
	Here, we show that if $u_n=n2^n\pm 1$, then the largest prime factor of $u_n\pm m!$ for $n\ge 0,~m\ge 2$ tends to infinity with $\max\{m,n\}$. In particular, the largest $n$ participating in the equation $u_n\pm m!=2^a3^b5^c7^d$ with $n\ge 1,~m\ge 2$ is $n=8$ for which $(8\cdot 2^8+1)-4!=3^4\cdot 5^2$.
	\end{abstract}
	
	\section{Introduction}\label{sec1}
	The numbers of the form $C_n=n2^n+1$ are called {\it Cullen} numbers.  	They were studied more than $100$ years ago by James Cullen.
	There are only $16$ known values of $n$ for which $C_n$ is prime. It is conjectured that there are infinitely many Cullen primes. Hooley \cite{Ho} proved that for most $n$, $C_n$ is composite. That is, the number of $n\le x$ such that $C_n$ is prime is $o(x)$ as $x\to\infty$.  Closely related to Cullen numbers are Woodall numbers of the form $W_n=n2^n-1$. 
	In \cite{GuLu}, the authors investigated Diophantine equations of the form $u_n\pm m!=s$, where $\{u_n\}_{n\ge 0}$ is a binary recurrent sequence of integers, and $s$ is an ${\mathcal S}$-unit, that is a positive integer whose prime factors are in a finite predetermined set of primes. In particular, they found all the Fibonacci numbers which can be written as a sum or difference between a factorial and a positive integer whose largest prime factor is at most $7$. In this paper, we revisit the above Diophantine equation but here $u_n$ is a Cullen or Woodall number. We note that both $\{C_n\}_{n\ge 0}$ and $\{W_n\}_{n\ge 0}$ are ternary recurrent sequences   of characteristic polynomial $(X-2)^2(X-1)=X^3-5X^2+8X-4$. Thus, to make our problem more general we take $\{u_n\}_{n\ge 0}$ to be a ternary recurrent sequence 
	whose characteristic polynomial has a double root. Let $f(X)=X^3-r_1X^2-r_2X-r_3$ be the characteristic polynomial of $\{u_n\}_{n\ge 0}$.
	Since it has a double root, it follows that all its roots are integers. We assume that $f(X)=(X-\alpha)^2(X-\beta)$, where $\alpha$ and $\beta$ 
	are integers. We admit that $\gcd(\alpha,\beta)=1$, which is equivalent to $\gcd(r_1,r_2,r_3)=1$ and  that it is nondegenerate. Thus, 
	$\alpha/\beta\ne \pm 1$. Then 
	$$
	u_n=p(n)\alpha^n+b\beta^n,\quad {\text{\rm where}} \quad p(X)=aX+c\in {\mathbb Q}[X],\quad a\ne 0.
	$$
	We have the following results. Let $P(m)$ be the largest prime factor of the nonzero integer $m$.
	\begin{thm}\label{thm1}
	Let $\alpha,\beta$ be coprime nonzero integers, $|\alpha|\ne |\beta|$, and $A$ be any nonzero integer. Then the estimate 
	$$
	P(u_n-Am!)\ge (1+o(1))\frac{\log n\log\log n}{\log\log\log n}
	$$
	holds as $n\to \infty$ uniformly in $m\ge 1$ such that $Am!\not\in \{b\beta^n, p(n)\alpha^n\}$.   
	\end{thm}
	\label{thm:1}
	Consider now a finite set of primes ${\mathcal P}=\{p_1,\ldots,p_k\}$ labelled increasingly, ${\mathcal S}$ 
	the set of all integers whose prime factors are in ${\mathcal P}$ and the Diophantine equation
	\begin{equation}
	\label{eq:(5)}
	u_n=A m!+Bs\quad {\text{\rm where}}\quad s\in {\mathcal S}\quad A,B\in {\mathbb Z},\quad \max\{|A|,|B|\}\le K.
	\end{equation}
	A solution is called non-degenerate if $Am!\not\in \{b\beta^n, p(n)\alpha^n\}$. 
	We have the following explicit version of Theorem \ref{thm:1}. 
	\begin{thm}\label{thm:2}
	Let 
	$$X=\max\{|u_0|,|u_1|,|u_2|,|r_1|,|r_2|,|r_3|,p_k,K,11\}.
	$$
	Then all nondegenerate solutions of equation \eqref{eq:(5)}  have $n<e^{12X}$. 
	\end{thm}
	As for Cullen and Woodall numbers, we have $(\alpha,\beta)=(2,1)$, $p(X)=X$ and $b\beta^n\in \{\pm 1\}$. Hence, the only degenerate solutions of equation \eqref{eq:(5)} in this case are the ones for which $Am!=b\beta^n\in \{\pm 1\}$, so $|A|=1$ and $m\in \{0,1\}$ (the case $Am!=n2^n$ only gives finitely many values for $n$, for example $n=1,3$ when $A=1$). Thus, if $m\ge 2$ and $n$ is large enough then the solutions are non-degenerate. Let ${\mathcal P}=\{2,3,5,7\}$. We have the following theorem.
	\begin{thm}\label{thm:3}
	If $P(n2^n\pm 1\pm m!)\le 7$ for some $m\ge 2$ then $n\le 8$ and $m\le 7$.  The solution with the largest $n$ is $(8\cdot 2^8+1)-4!=3^4\cdot 5^2$ and the solution with the largest $m$ is $4\cdot 2^4-1+5040=3^6\cdot 7$.
	\end{thm}
	The full set of solutions is given at the end of the paper. 
		\section{preliminaries}\label{sec2}
		We start by recalling some basic notions from height theory. The absolute logarithmic height $h(\eta)$ of an algebraic number $\eta$  is given by the formula  
		\begin{equation}\label{eqn2.1}
		h(\eta)=\frac{1}{d(\eta)}\left(\log |a_0| +\sum _{i=1}^{d(\eta)} \log \left(\max \{|\eta ^{(i)}|,1\}\right)\right),
		\end{equation}
		where $d(\eta)$ is the degree of $\eta$ over $\mathbb{Q}$, and 
		\begin{equation}\label{eqn2.2}
		f(X)=a_0\prod _{i=1}^{d(\eta)} \left(X-\eta ^{(i)}\right)\in \mathbb{Z}[X]
		\end{equation}
		is the minimal polynomial of $\eta$  of degree $d(\eta)$ over $\mathbb{Z}$. We use the following properties of the absolute logarithm height function $h(\cdot)$:
		\begin{lem}\label{H}
			Let $\eta, ~\gamma$ be the algebraic numbers. Then we have:
			\begin{itemize}
				\item $	h(\eta \pm \gamma)\leq h(\eta) + h(\gamma) + \log 2,$
				\item $	h(\eta  \gamma^{\pm 1})\leq h(\eta) + h(\gamma),$
				\item $	h(\eta^s )= \vert s \vert  h(\eta) \quad (s \in \mathbb{Z}).$
			\end{itemize}
		\end{lem}
Let $\mathbb{K}$ be a  number field of degree $D$ over $\mathbb{Q}$ embedded in $\mathbb{C}$.  Let   
$\eta_1,\ldots,\eta_l\in \mathbb{K}$ not $0$ or $1$ and $d_1,d_2,\ldots,d_l \in \mathbb{Z}^*.$ We put $B^*=\max\{\vert d_1 \vert,\ldots, \vert d_l \vert, 3 \}$. We take $A_j\geqslant \max\{Dh(\eta_j),\vert \log(\eta_j)\vert,0.16\}~(1\leqslant j\leqslant l)$, $\omega=A_1A_2\cdots A_l$,  and 
$$\Lambda =\eta_1^{d_1}\cdots \eta_l^{d_l}-1.
$$
		
		The following lemma is a consequence of Matveev's theorem \cite{M}.
		
		\begin{lem}[See Theorem 9.4 in \cite{Bu}] \label{BMS}
			If $\Lambda \neq 0$ and $\mathbb{K}\subseteq \mathbb{R}$,  then 
			$$\log(\vert \Lambda \vert) > -1.4\times 30^{l+3}l^{4.5}D^2\omega \log(eD)\log(eB^*).
			$$
		\end{lem}
		A $p$-adic analogue of Matveev's theorem is due to  Yu \cite{Y}. Here we recall this result.  Let $\pi$ be a prime ideal in the ring $\mathcal{O}_{\mathbb{K}}$ of algebraic integer in $\mathbb{K}$. Let $e_{\pi}$ and $f_{\pi}$ be respectively  the ramification index and the inertial degree of $\pi$. Let $p$ be the prime number above $\pi$ and $\nu_{\pi}(\eta)$ be the order at  which  $\pi$ appears in the prime factorization of the principal fractional ideal $\eta \mathcal{O}_{\mathbb{K}}$.  
\begin{lem}[Yu \cite{Y}] \label{Y}
Let $H_j \geq \max \{h(\eta_j), \log p \}$, for $j = 1, 2, \ldots, l$. 	If $\Lambda \neq 0$,  then 
$$\nu_{\pi}(\vert \Lambda \vert) \leq 19 (20 \sqrt{l+1}D)^{2(l+1)}e_{\pi}^{l-1}\dfrac{p^{f_{\pi}}}{(f_{\pi}\log p)^2}\log(e^5lD)H_1\cdots H_l \log(B^*).
$$
\end{lem}	
		
		The following result is well-known and can also be proved using the fact that the map $x \mapsto x/ (\log x)^m$ is increasing when $x> e^m$ and $m\geq 1$.	
		
		\begin{lem}\label{3}
			If $s\geq 1$,  $T>(4s^2)^s$ and $T> x/ (\log x)^s$, then
			$$x< 2^sT (\log T)^s.
			$$
		\end{lem}	
		
	\section{Bounds}
	Put 
	$$
	Y:=\max\{|r_1|,|r_2|,|r_3|,|u_0|,|u_1|,|u_2|\}.
	$$
	Note that $Y\ge 3$. Indeed, since by the Viete relations $\alpha^2\beta=r_3$, we immediately get that $Y\ge 3$ except for the cases $\alpha=\pm 1$ and $\beta=\pm 2$. Calculating all $(X\pm 1)^2(X\pm 2)$, we get that there is always a coefficient which is at least $3$ in absolute value.  
	The following parallels Lemma 8 in \cite{GuLu}.
	\begin{lem}
	\label{lem:heights}
	We have $\max\{|\alpha|,|\beta|\}\le Y$. Further, $a,b,c$ are rational numbers of numerators at most $4Y^3$ and denominators at most $Y^3$. In particular, 
	$$\max\{h(a),h(b),h(c)\}\le \log(4Y^3).
	$$  
	\end{lem}
	\begin{proof}
	By the Viete relation, we have $\alpha^2\beta=r_3$. Hence, $|\alpha|^2|\beta|\le Y$, which implies that $|\alpha|\le Y^{1/2}$ and $|\beta|\le Y/|\alpha|^2$. As for $a,b,c$, we solve the linear system
	\begin{eqnarray*}
	c+b & = & u_0;\\
	a\alpha+c\alpha+b\beta & = & u_1;\\
        2a\alpha^2+c\alpha^2+b\beta^2 & = & u_2.
        \end{eqnarray*}
        We solve it with Cram\'er's rule getting that $a,b,c$ are of the form $\Delta_i/\Delta$, where 
        $$
        \Delta=\left| \begin{matrix} 0 & 1 & 1 \\
        \alpha & \alpha & \beta\\
        2\alpha^2 & \alpha^2 & \beta^2
        \end{matrix}
        \right|,
        $$
        $$
        |\Delta|=|\alpha||\alpha^2-2\alpha\beta+\beta^2|\le 4\max\{|\alpha|^3,|\alpha||\beta|^2\}\le \max\{4Y^{3/2},4Y^2/|\alpha|\}.
        $$ 
        Since $Y\ge 3$, the expression on the right above is always $\le Y^3$ except in the case when $|\beta|>|\alpha|=1$. In this last case, we have 
        $|r_2|=|2\beta\pm 1|$, so $|\beta|\le (Y+1)/2$, therefore 
        $$
        |\Delta|=(\beta\pm 1)^2<4\beta^2\le 4\left(\frac{Y+1}{2}\right)^2=(Y+1)^2<Y^3.
        $$
        Thus, the inequality $|\Delta|\le Y^3$ always holds. 
         $\Delta_i$ are $3\times 3$ determinants obtained replacing some column in $\Delta$ by $(u_0,u_1,u_2)^T$; i.e.,  
         $$
         \Delta_1=\left| \begin{matrix} u_0 & 1 & 1 \\
         u_1 & \alpha & \beta\\
         u_2 & \alpha^2 & \beta^2
         \end{matrix}
         \right| = u_0(\alpha\beta^2-\alpha^2\beta)-(u_1\beta^2-u_2\beta) + \alpha^2 u_1- u_2 \alpha,
         $$
         $$
         \Delta_2=\left| \begin{matrix} 0 & u_0 & 1 \\
         \alpha & u_1 & \beta\\
         2\alpha^2 & u_2 & \beta^2
         \end{matrix}
         \right| = -u_0(\alpha\beta^2-2\alpha^2 \beta)+ \alpha u_2-2\alpha^2 u_1
         $$ and 
         $$
         \Delta_3=\left| \begin{matrix} 0 & 1 & u_0 \\
         \alpha & \alpha & u_1\\
         2\alpha^2 & \alpha^2 & u_2
         \end{matrix}
         \right| = -(\alpha u_2- 2\alpha^2u_1) - u_0\alpha^3.
         $$
         Using the fact that $\vert \alpha^2\beta \vert  \leq Y$, we deduce  $|\Delta_i|\leq 4Y^3$. Let us make these deductions. We have 
         $$
         |\Delta_1|\le 6Y\max\{|\alpha\beta^2|,|\alpha^2\beta|\}.
         $$
        If the maximum is in $|\alpha^2\beta|$, we then get $|\Delta_1|\le 6Y^2<4Y^3$. Otherwise, 
        $$
        |\Delta_1|\le 6Y|\alpha \beta^2|\le 6Y|\alpha|(Y/|\alpha|)^2=6Y^3|\alpha|^{-1}<4Y^3\quad {\text{\rm if}}\quad |\alpha|>1.
        $$ 
        Finally, if $\alpha=\pm 1$, then by the Hadamard inequality 
        $$
        |\Delta_1|\le ({\sqrt{3}}Y)\times {\sqrt{3}} \times {\sqrt{1+\beta^2+\beta^4}}\le 3Y\beta^2\left(1+\frac{1}{4}+\frac{1}{16}\right)^{1/2}<4Y^3,
        $$
        where the last inequality holds because $3(1+1/4+1/16)^{1/2}<4$. The argument is similar for $|\Delta_2|$. Namely,
        $$
        |\Delta_2|\le 6Y\max\{|\alpha \beta^2|,|\beta\alpha^2|\}.
        $$
        As in the previous analysis, the expression in the right above is at most $4Y^3$ except possibly if $\alpha=\pm 1$ in which case by analysing the expression for $\Delta_2$ directly we get
        $$
        |\Delta_2|\le Y(|\beta|^2+2|\beta|+3)=Y\beta^2\left(1+\frac{2}{|\beta|}+\frac{3}{|\beta|^2}\right)\le Y^3\left(1+\frac{2}{2}+\frac{3}{2^2}\right)<4Y^3.
        $$
        Finally, 
        $$
        |\Delta_3|\le 6Y\alpha^3\le 6Y(Y^{1/2})^3=6Y^{5/2}<4Y^3
        $$
       since $Y\ge 3$. Hence, $\max\{h(a),h(b),h(c)\}\le \log(4Y^3)$.  
	\end{proof}
	The following parallels Lemma 9 in \cite{GuLu}. 
	\begin{lem}
	\label{lem:zero}
	If $u_n=0$, then $n<39Y\log Y$. 
	\end{lem}
	\begin{proof}
	We assume $n\ge 39Y\log Y$ in order to get a  contradiction. The  relation $u_n=0$ implies
	\begin{equation}
	\label{eq:1}
	\left(\frac{|\beta|}{|\alpha|}\right)^n=\frac{|p(n)|}{|b|}.
	\end{equation}
	Assume $|\beta|>|\alpha|$. Then $|\alpha|<Y^{1/3}$. In the right--hand side above, the fraction $|p(n)|/|b|$ is, by Lemma \ref{lem:heights}, at most 
	\begin{equation}
	\label{eq:2}
	\frac{4Y^3(n+1)}{1/Y^3}\le 4Y^6(n+1)<n^7.
	\end{equation}
	For the last inequality above, we used that $n>39 Y\log Y$ and $Y\ge 3$. Taking logarithms, we get
	$$
	n\log\left(\frac{ |\beta|}{|\alpha|}\right)<7\log n.
	$$
	In the left, 
	$$\log\left(\frac{|\beta|}{|\alpha|}\right)\ge \log\left(1+\frac{1}{|\alpha|}\right)>\log\left(1+\frac{1}{Y^{1/3}}\right)>\frac{1}{Y^{1/3}+1}>\frac{1}{Y},
	$$
	so we get $n<7Y\log n$. By Lemma \ref{3} with $s=1$, we get 
	$$
	n<14Y(\log7+\log Y)\le 14Y(\log Y)\left(\frac{\log 7}{\log 3}+1\right)<39Y\log Y, 
	$$
	a contradiction. If $|\beta|<|\alpha|$, we then get
	$$
	\left(\frac{|\alpha|}{|\beta|}\right)^n=\frac{|b|}{|p(n)|}.
	$$
	The numerator in the right--hand side above is a rational number of denominator  a divisor of $|\Delta|\le Y^3$, so $|b|\le 4Y^3$ and $|p(n)|\ge 1/Y^3$. Thus, the right--hand side above is at most $4Y^6$. We thus get
	\begin{equation}
	\label{eq:3}
	n\log\left(\frac{|\alpha|}{|\beta|}\right)<\log(4Y^6).
	\end{equation}
	In the left--hand side, since now $|\beta|< Y^{1/3}$, we have
	$$
	\log\left(\frac{|\alpha|}{|\beta|}\right)\ge \log\left(1+\frac{1}{|\beta|}\right)> \frac{1}{|\beta|+1}> \frac{1}{Y^{1/3}+1}> \frac{1}{Y}.
	$$
	Hence, we get
	\begin{equation}
	\label{eq:4}
	n<Y\log(4Y^6)<Y\log(Y^8)=8Y\log Y<39Y\log Y.
	\end{equation}
	\end{proof}
	We need a lower bound on $u_n$. 
	\begin{lem}
	\label{lem:lowerboundun}
	Assume $n>Y^{8}$. If $|\beta|>|\alpha|$,  then 
	$$
	|u_n|>\frac{|\beta|^n}{2Y^3}.
	$$
	Otherwise, that is if $|\alpha|>|\beta|$, then 
	$$
	|u_n|>\frac{n|\alpha|^n}{6Y^3}.
	$$
	\end{lem}
	\begin{proof}
	Suppose $|\beta|>|\alpha|$. We start by showing that 
	$$
	|b||\beta|^n>2|p(n)| |\alpha|^n
	$$
	If this is not so, then 
	$$
	\left(\frac{|\beta|}{|\alpha|}\right)^n\le \frac{2|p(n)|}{|b|}.
	$$
	We now follow the previous argument. The only difference is that in estimate \eqref{eq:1} the right--hand side is twice as large so it is bounded by $8Y^6$.  So, the bound of \eqref{eq:2} is now 
	$$
	\frac{8Y^3(n+1)}{1/Y^3}=8Y^6(n+1)<n^2,
	$$
	which holds since $n>Y^{8}>39Y\log Y$ for $Y>3$. In particular, estimate \eqref{eq:2} holds and we saw that it leads to $n<39Y\log Y$, a contradiction. Hence, 
	$$
	|u_n|=|p(n)\alpha^n+b\beta^n|\ge |b||\beta|^n-|p(n)||\alpha|^n\ge 0.5|b|\beta^n\ge \frac{|\beta|^n}{2Y^3}.
	$$
	Assume next that $|\alpha|>|\beta|$. Then, as in the previous argument, we show that 
	$$
	|p(n)| |\alpha|^n>2|b| |\beta|^n.
	$$
	Indeed, for if not, then 
	$$
	\left(\frac{|\alpha|}{|\beta|}\right)^n\le \frac{2|b|}{|p(n)|}.
	$$
	We follow again the previous argument. The argument of the logarithm in the right--hand side of \eqref{eq:3} is now $8Y^3$ so now the analogue of  \eqref{eq:4} becomes 
	$$
	n<Y\log(8Y^6)<Y\log (Y^{8})=8Y\log Y<39Y\log Y,
	$$
	again  a contradiction. In the above, we used that $Y^3>10$ since $Y\ge 3$.  Hence,
	\begin{equation}
	\label{eq:5}
	|u_n|=|p(n)\alpha^n+b\beta^n|> 0.5|p(n)|\alpha^n.
	\end{equation}
	In the right--hand side, we have
	$$
	|p(n)|=|an+c|=n|a|\left|1+\frac{c}{an}\right|.
	$$
	Since $c$ is at most $4Y^3$ and $|a|$ is at least $1/Y^3$, it follows that $|c|/|an|\le 4Y^6/n<2/3$ since $n>Y^{8}$ and $Y^2\ge 9>6$. Hence,
	$$
	|p(n)|> \frac{n|a|}{3}> \frac{n}{3Y^3},
	$$
	which combined with \eqref{eq:5} gives the desired lower bound on $|u_n|$. 
	\end{proof}
	The following parallels Lemma 10 in \cite{GuLu}.
  \begin{lem}
  \label{lem:muun}
  	Let $p$ be a prime number. If $n>X^8$, then 
	$$
  	\nu_p(u_n)\leq 1.2 \times 10^{12} \dfrac{p}{\log p}( \log p + \log Y) \log^2n.
  	 $$
  \end{lem}	
\begin{proof}
First of all note that $u_n\ne 0$ in the range $n>X^8$ by Lemma \ref{lem:zero}.
 Assume that $p$ does not divide $\alpha$. Then 
 $$\nu_p(u_n)= \nu_p(p(n))+ \nu_p(1-(-b)\beta^np(n)^{-1}\alpha^{-n}).$$  Since
 $$
 p^{\nu_p(p(n))}\leq |{\text{\rm numerator}}(an + b)|\leq 4Y^3(n+1),
 $$ 
 we deduce that $$\nu_p(p(n))\leq  \dfrac{\log (4Y^3) + \log (n+1)}{\log p}.$$ On the other hand, we apply Lemma \ref{Y} with the following parameters: $$\Lambda:=\vert 1+b\beta^np(n)^{-1}\alpha^{-n} \vert
 $$
   $\eta_1:=(-b)^{-1}p(n),~~\eta_2:=\beta\alpha^{-1},~~d_1:=-1,~~d_2:=n.$ Further, 
  $$
  h(p(n)b^{-1})\leq h(p(n))+h(b)\leq \log (16Y^6(n+1))
  $$ and
  $$
  h(\beta \alpha^
  {-1})\leq h(\alpha)+ h(\beta)\leq \log Y.
  $$ 
  Applying Lemma \ref{Y}, we get
  \begin{eqnarray*}
  \nu_p(\vert \Lambda \vert) &\leq & 19(20\sqrt{3})^6\cdot \dfrac{p}{(\log p)^2}\log(2e^5)(\log p + \log Y)(\log p+\log (16Y^6(n+1))) \log n\\ 
  &\leq & 19(20\sqrt{3})^6\cdot \dfrac{p}{(\log p)^2}\log(2e^5)(\log p + \log Y)(\log p+2.33\log n) \log n\\
  &\leq & 19(20\sqrt{3})^6\cdot \dfrac{p}{\log p}\log(2e^5)(\log p + \log Y)\left(1+\frac{2.33}{\log p}\right)(\log n)^2\\
  & \leq & 19(20{\sqrt{3}})^6\cdot 4.4 \log(2e^5) \dfrac{p}{\log p}(\log p + \log Y)(\log n)^2\\
  &\leq & 1.1 \times 10^{12}\cdot \dfrac{p}{\log p}(\log p + \log Y)(\log n)^2.
 \end{eqnarray*}
 The only fact to justify in the above calculation  is that $2.33\log n>\log(16Y^6(n+1)$, but this is so since 
 $$
 \frac{n^{2.33}}{n+1}>(n-1) n^{0.33}\ge (X^8-1)(X^{8\cdot 0.33})>(X-1)(X^7+\cdots+1) X^2>(2\cdot 3^2)X^7>16Y^6,
 $$
 since $X\ge Y\ge 3$. Hence, 
 $$ \nu_p(u_n) \leq \nu_p(p(n))+\nu_p(\Lambda)< 1.2 \times 10^{12}\cdot \dfrac{p}{\log p}(\log p + \log Y)(\log n)^2.
 $$
If $p$ divides $\alpha$, then 
$$
\nu_p(u_n)\le \min\{\nu_p(p(n)\alpha^n),\nu_p(b\beta^n)\}=\min\{\nu_p(p(n)\alpha^n),\nu_p(b)\}\le \log(4Y^3)/\log p,
$$
a much better inequality.  
\end{proof}

\section{Proof of Theorems \ref{thm:1} and \ref{thm:2}}\label{sec5}
In this section, we  prove Theorems  \ref{thm:1} and \ref{thm:2} simultaneously. Let $\pi (X)$ be the number of primes $p\leq X.$ By the prime number theorem, we have $$\pi (X)= (1+o(1))\dfrac{X}{\log X }
$$ as $X\rightarrow \infty$. To prove these theorems, it suffices to show that 
\begin{equation*}
n \leq  \left\{
\begin{aligned}
e^{12X}  & \quad \mbox{for all}\quad X  \\
M(X)^{(1+o(1))}  & \quad \mbox{when}\quad X \rightarrow \infty,
\end{aligned}
\right.
\end{equation*}
where $$M(X)=e^{\pi (X)\log \log X}.
$$
Since $X\geq 11,$ we may assume that $n\geq X^8$, otherwise we directely have the desired result. 

\medskip

\textbf{Case 1}: $A=0$.  In this situation, equation \eqref{eq:(5)} becomes $u_n=Bs,$ with $s=p_1^{\theta_1}p_2^{\theta_2}\cdots p_k^{\theta_k}$. Using the fact that 
$\theta_i\leq \nu_{p_i}(u_n)$, we deduce via Lemma \ref{lem:muun}, 
\begin{eqnarray}
\label{eq:theta}
\theta_i &\leq & 1.2 \times 10^{12}\cdot \dfrac{p_i}{\log p_i}(\log p_i + \log Y)(\log n)^2\nonumber\\ 
&\leq & 1.2 \times 10^{12}\cdot \dfrac{X}{\log X}(\log X + 2\log X)(\log n)^2\nonumber\\
&\leq & 3.6 \times 10^{12} X(\log n)^2, 
\end{eqnarray}
where, for the second inequality, we used the fact that the function $x\mapsto x/\log x$ is increasing for all $x>e$ and $X\ge 11$.  We have 
\begin{eqnarray}
\label{eq:10}
\log \vert u_n \vert &=& \log B + \sum_{i=1}^{k} \theta_i \log p_i\nonumber\\ 
&\leq& \log X + 3.6\pi (X)\cdot 10^{12} X\log X (\log n)^2\nonumber\\
&\leq& 1.3\times 3.6 \cdot 10^{12} X^2(\log n)^2\nonumber\\
&\leq& 5 \times 10^{12} X^2(\log n)^2,
\end{eqnarray}
where, for the second inequality, we used the fact that $\pi(X)\leq 1.25 X/ \log X$ (see Corollary 2 in \cite{8}). When $\vert \beta \vert > \vert \alpha \vert$, by Lemma \ref{lem:lowerboundun}, we have  
$$
n \log \vert \beta \vert - \log(2Y^3) \leq \log \vert u_n \vert,
$$
which combined with \eqref{eq:10} and the fact that $|\beta|\ge 2$ implies 
\begin{equation}
\label{eq:7.3}
n\leq 7.3 \times 10^{12} X^2 (\log n)^2.
\end{equation}
On the other hand, if $\vert \alpha \vert > \vert \beta \vert$, then 
using again Lemma \ref{lem:lowerboundun},  we get 
$$ 
\log n + n \log \vert \alpha \vert - \log (6Y^3)\leq 
5\cdot 10^{12}X^2(\log n)^2,
$$
which also implies \eqref{eq:7.3}.
Applying Lemma \ref{3} with $s=2$ and $T:=7.3\times 10^{12} X^2$ to \eqref{eq:7.3}, we get 
\begin{align*}
n &\leq 4 \times 7.3 \times 10^{12} X^2(\log (7.3 \times  10^{12} X^2))^2\\ 
&= 29.2 \times 10^{12} X^2 (\log (7.3 \times  10^{12})+ 2\log X)^2\\
&\leq  29.2 \times 10^{12} X^2 (\log X)^2\left( \dfrac{\log (7.3 \times 10^{12})}{\log 11}+2 \right) ^2\\
&\leq 6.1\times  10^{15}X^2(\log X)^2.
\end{align*}
The last expression above is less than $e^{12X}$ as $X\geq 11$ and it is certainly $ M(X)^{o(1)}$ as $X\rightarrow \infty$.

\medskip
 
 \textbf{Case 2}:  $B=0$. Our equation becomes $u_n=Am!$. This implies that 
 $$
 \log \vert u_n \vert \le  \log X + m \log m.
 $$ 
 If  $\vert \beta \vert > \vert \alpha \vert$, then by Lemma \ref{lem:lowerboundun}, we have $ n \log \vert \beta \vert - \log (2Y^3)\leq \log \vert u_n \vert $ which gives 
 $$  n \log \vert \beta \vert \leq \log 2 + 3 \log X + \log X + m\log m
 $$
 so 
 \begin{equation}
 \label{1'}
 n \log 2 \leq  n \log \vert \beta \vert \leq 5 \log X + m \log m. 
 \end{equation}
 If $\vert \alpha \vert > \vert \beta \vert$, then again by Lemma \ref{lem:lowerboundun}, we have $\log n + n\log \vert \alpha \vert - \log (6Y^3) \leq \log \vert u_n \vert.$ Again since $\log \vert u_n \vert \leq \log X + m\log m,$ 
 it follows that 
$$ 
n \log 2 \leq \log 6 + 3 \log Y + \log X + m\log m<5\log X+m\log m.
$$
So, relation \eqref{1'} holds independently of which of $|\alpha|$ or $|\beta|$ is larger.
 Since $n> X^8,$ it follows that $0.001n> \log X^5$, so \eqref{1'} gives $m\log m>0.69 n$. Hence, 
 $$ 
 m > \dfrac{0.59n}{\log (0.59n)}> \dfrac{0.59n}{\log n}.
 $$ 
 If $m\leq 10$, then $n/ \log n \leq 17$ which implies $n\leq 73<e^{11X}$, a very good bound. If $m>10$, we have 
 $$ \nu_2(m!)= \left\lfloor \frac{m}{2} \right\rfloor + \left\lfloor \frac{m}{4}\right\rfloor + \cdots  >  \frac{m}{2} \geq \frac{0.2n}{ \log n}.
 $$
Since certainly $ \nu_2(m!) \leq \nu_2(u_n)$, we get, by estimate \eqref{eq:theta}, 
$$
\frac{0.2 n}{\log n}<3.6\times 10^{12} X(\log n)^2,
$$
which implies 
$$ n \leq 10^{14}  X (\log n)^3.
$$ Applying Lemma \ref{3} with $s=3,$ one obtains:
\begin{eqnarray*}
n &\leq & 8 \times 10^{14}  X (\log(10^{14}  X))^3\\ 
&\leq & 8 \times 10^{14}  X (\log (10^{14})+ \log X)^3\\
&\leq  & 8 \cdot 10^{14} X(\log X)^3\left( \frac{\log (10^{14})}{\log 11}+1 \right) ^3\\
&\leq & 2,5 \cdot 10^{18}X(\log X)^3,
\end{eqnarray*}
The last expression above is less than $e^{12X}$  since $X\geq 11$ and is certainly $ M(X)^{o(1)}$  when $X\rightarrow \infty$.

\medskip

\textbf{Case 3}: {$AB \neq 0$}. We have $u_n=Am!\pm Bs$, which implies 
$$ 
p(n)\alpha^n + b\beta^n= Am!\pm Bs.
$$ 
Put $ \gamma  = \max \{\vert \alpha \vert , \vert \beta \vert  \}.$ If $m! <  \gamma  ^{n/2}$ and $\vert \beta \vert $ realizes the maximum ($ \gamma $= $\vert \beta \vert$),  then one has:
\begin{align*}
\vert 1 \mp Bsb^{-1}\beta^{-n} \vert & \leq \left| \dfrac{p(n)}{b } \right|  \cdot \left| \dfrac{\alpha}{\beta} \right| ^n  +  \left| \dfrac{A}{b } \right|  \cdot \left| \dfrac{1}{\sqrt{|\beta|}} \right| ^{n}.
\end{align*}
We know that  $b^{-1}\leq \vert \Delta \vert \leq Y^3 $ and $\vert p(n) \vert \leq 4Y^3(n+1)$. 
However,
$$ \left| \dfrac{\alpha}{\beta} \right| ^n \leq \left(\dfrac{1}{1+1/(\vert \beta \vert -1)} \right) ^n \leq \left(\dfrac{1}{1+1/(Y -1)} \right) ^n \leq \left(\dfrac{1}{e^{1/Y}} \right) ^n,
$$ 
where we used the fact that $\vert \alpha \vert \neq \vert \beta \vert$, $\vert \beta \vert < Y$ and $(1+1/(x-1))^x > e$ for all $x>2$. Since $Y\geq 3$, we have
$$ 
e^{1/Y} \leq e^{1/3} \leq \sqrt{2} \leq \sqrt{ \vert \beta \vert}.
$$ 
Hence, we deduce that  
\begin{equation}
\label{a'}
\vert 1 \pm Bsb^{-1}\beta^{-n} \vert \leq \dfrac{4X^6(n+1)+X^4}{e^{n/Y}} \leq  \dfrac{4X^6(n+2)}{e^{n/Y}}\leq  \dfrac{X^7n}{e^{n/Y}},
\end{equation}
where, for the last inequality, we used the fact that $n+2\leq 1.5n$ and $X>6$.
We also know that  $$ \dfrac{p(n) \alpha^n+ b\beta^n-Am!}{B}= \pm s
$$ which implies
\begin{eqnarray}
\label{eq:gamma}
2^{\max \theta_i}&\leq &\left| \dfrac{p(n)}{B} \right|\cdot \vert \alpha \vert ^n + \left| \dfrac{b}{B} \right|\cdot \vert \beta \vert ^n + \left| \dfrac{A}{B} \right|\cdot (\sqrt{ \gamma })^n\nonumber \\ 
&\leq &4X^3(n+1) \vert \beta \vert ^n +4X^3\vert \beta \vert ^n + X\vert \beta \vert ^n\nonumber\\
&\leq &X^3n\vert \beta \vert ^n\nonumber\\
&\leq & n^2 \gamma ^n,
\end{eqnarray}
where, we used the fact that $ \gamma $= $\vert \beta \vert$ and for the last inequality, the fact that $X>11$ and $n\geq X^8$. Hence we deduce that 
\begin{equation}
\label{eq:max}
 \max \theta_i \leq \frac{1}{\log 2}(2\log n + n \log X) \leq 4n\log X \leq n^{3/2},
\end{equation}
where, for the last inequality, we used the fact that $n\geq X^8.$\\ 
 Now, we apply Matveev's Theorem with the following parameters: 
 \begin{equation}
 \label{eq:param}
 l=k+2,\quad \eta_1=Bb^{-1},~ \eta_2=\gamma,~\eta_{2+i}=p_i,~d_1=1,~d_2=-n,~d_{2+i}=\theta_i, \quad i=1,\ldots,k.
 \end{equation}
 By Lemma \ref{H}, one has 
 $$ 
 h(\eta_1)\leq h(B)+h(b)\leq  5\log X,~ h(\eta_2)\leq \log X,~h(\eta_{2+i})\leq \log X,\quad i=1,\ldots,k.
 $$  
 Put $\vert \Lambda\vert = \vert 1 \mp Bsb^{-1}\beta^{-n} \vert$. Note that $\Lambda\ne 0$ since we are working with non-degenerate solutions. Applying Lemma \ref{BMS}, we get 
\begin{eqnarray}
\label{eq:lowLambda}
\log \vert \Lambda \vert &\geq & -1.4 \cdot 30^{k+5}(k+2)^{4.5}\cdot 5\log X \cdot \log X \cdot (\log X)^k\cdot (1+(3/2)\log n) \nonumber \\ 
&\geq &  -11.2\cdot 30^{k+5}(k+2)^{4.5}(\log X)^{k+2} \log n,
\end{eqnarray}
where, for the last inequality, we used the fact that $n>X^{8}\ge 11^8$.  
Using the relation \eqref{a'}, we deduce that 
$$ 
n/Y-\log (X^7n) \leq 11.2\cdot 30^{k+5}(k+2)^{4.5}(\log X)^{k+2}\cdot \log n, 
$$ 
which implies 
\begin{equation}
\label{eq:5.1}
n \leq 11.3 \cdot 30^{k+5}(k+2)^{4.5}X(\log X)^{k+2}\cdot \log n. 
\end{equation} 
Applying Lemma \ref{3} to \eqref{eq:5.1} with $s=1$, one obtains 
\begin{align*}
n &\leq 2\cdot 11.3\cdot 30^{k+5}X^{5.5}(\log X)^{k+2}(\log ( 11.3\cdot 30^{k+5}X^{5.5}(\log X)^{k+2}))\\ 
&\leq  22.6\cdot 30^{k+5}X^{5.5}(\log X)^{k+3}(1+(3/2)(k+5)+5.5+k+2)\\
&\leq   22.6\cdot 30^{k+5}X^{5.5}(\log X)^{k+3}(16+(5/2)k)\\
&\leq 22.6\cdot X^{5.5}(30\log X)^{k+5}(16+(5/2)k),
\end{align*}
where, for the second inequality, we used the fact that $\log (11.3) < 1.02 \cdot \log X$,  $\log X<X$ and $\log 30\leq (3/2)\log X$.
Since $k\leq \pi(X) \leq 1.25X/ \log X< 0.53X$ for $X\ge 11$, it follows that 
\begin{equation}
\label{eq:6}
n\leq 2.2 \cdot 30 X^{6.5}(30\log X)^{k+5}<X^{6.5} (30\log X)^{k+6}.
\end{equation}
The last expression is at most $ M(X)^{(1+o(1))}$  when $X\rightarrow \infty$.  Notice that the logarithm of the right hand side of the above relation 
$$6.5 \log X + (k+6)\log (30\log X).
$$
Since $k<1.25X/\log X$, one proves easily that the right--hand side above is smaller than $8X$ for $X\ge 11$. 
So we have the desired result.\\

Let's assume now that  $\vert \alpha \vert $ realizes the maximum ($ \gamma $= $\vert \alpha \vert$) and $m! <  \gamma  ^{n/2}$.  We then have $p(n)\alpha^n \mp Bs = Am! -b\beta^n$, which implies that 
\begin{equation}
\label{b'}
\vert 1 \mp Bsp(n)^{-1}\alpha^{-n} \vert \leq \dfrac{X^4}{\vert \alpha \vert ^{n/2}}+ 4X^6 \cdot \left|  \frac{\beta}{\alpha}\right|^n,
\end{equation}
 where we used the fact that $\vert p(n) \vert ^{-1}<X^3$.
Further, we get $$ \left| \dfrac{\beta}{\alpha} \right| ^n \leq \left(\dfrac{1}{1+1/(\vert \alpha \vert -1)} \right) ^n \leq \left(\dfrac{1}{1+1/(Y -1)} \right) ^n \leq \left(\dfrac{1}{e^{1/Y}} \right) ^n.
$$  Since $Y\geq 3,$ we have
$$ 
e^{1/Y} \leq e^{1/3} \leq \sqrt{2} \leq \sqrt{ \vert \alpha \vert}.
$$ 
Hence, we deduce that  
\begin{equation}
\label{2''}
\vert 1 \mp Bsp(n)^{-1}\alpha^{-n} \vert \leq   \dfrac{5X^6}{e^{n/Y}}<\frac{X^7n}{e^{n/Y}},
\end{equation} 
which is the same as \eqref{a'}. We also have $$ \dfrac{p(n) \alpha^n+ b\beta^n-Am!}{B}= \pm s,
$$
which implies that 
\begin{align*}
2^{\max \theta_i}&\leq \left| \dfrac{p(n)}{B} \right|\cdot \vert \alpha \vert ^n + \left| \dfrac{b}{B} \right|\cdot \vert \beta \vert ^n + \left| \dfrac{A}{B} \right|\cdot (\sqrt{\vert \alpha \vert})^n \\ 
&\leq 4X^3(n+1) \vert \alpha \vert ^n +4X^3\vert \alpha \vert ^n + X\vert \alpha \vert ^n\\
&\leq n^2\vert \gamma \vert ^n,
\end{align*}
where we used the fact that $ \gamma =\vert \alpha \vert$ and for the last inequality the fact that $X>11$ and $n\geq X^8$.
The above inequality is the same as \eqref{eq:gamma}. Thus, \eqref{eq:max} also holds. 
Now, we apply Matveev's Theorem to the left--hand side of \eqref{2''} with the following parameters:
$l=k+2,~\eta_1=Bp(n)^{-1},~ \eta_2=\alpha,~\eta_{2+i}=p_i,~d_1=1,~d_2=-n,~d_{2+i}=\theta_i,~i=1,\ldots,k.
$
The fact that this is nonzero follows since we are working with nondegenerate solutions. 
This is the same as \eqref{eq:param} with the exception of $\eta_1$ for which 
$$
h(\eta_1)\le h(B)+h(p(n))\le \log X+\log(4X^3)+\log(n+1)<5\log X+\log(n+1)<2\log n.
$$
Thus, the same calculation as before shows that we have a bound as in \eqref{eq:lowLambda} except that $2\log X$ has been swamped by $\log n$. We leave the same power of $\log X$ in the right--hand side and just record that we have a lower bound as in \eqref{eq:lowLambda} with an additional $\log n$ factor in the right--had side: 
$$
\log \vert \Lambda \vert \geq -5\cdot 30^{k+5}(k+2)^{4.5}(\log X)^{k+2}(\log n)^2. 
$$
The previous calculation involving \eqref{2''} (which implies \eqref{a'}) leads to \eqref{eq:5.1} with an additional $\log n$ in the right--hand side:
$$ 
n \leq 5.1 \cdot 30^{k+5}(k+2)^{4.5}X(\log X)^{k+2}(\log n)^2. 
$$
Applying Lemma \ref{3} with $s=2$, one obtains 
\begin{eqnarray*}
n &\leq & 20.4 (30\log X)^{k+5}X(k+2)^{4.5}(\log 5.1+\log X + (k+5)\log(30\log X)+ 4.5 \log (k+2))^2\\ 
&\leq  & 20.4 (30)^{-2} (30\log X)^{k+7}X(k+2)^{4.5}\left( \frac{\log 5.1}{\log 11}+1+2(k+5)+4.5  \right)^2 \\
&\leq & 20.4 (30)^{-2} (30\log X)^{k+7}X(k+2)^{4.5}\left(16.2+2k  \right)^2 \\
&\leq &  20.4 (30)^{-2} (30\log X)^{k+7} X^{7.5} \left(0.53+2/11\right)^2\left(2\cdot 0.53+16.2/11\right)^2\\
& < & X^{7.5} (30\log X)^{k+7}.
\end{eqnarray*}
where we used the fact that $k\leq \pi(X)\leq 1.25X/\log X<0.53X$ (so $k+2<X$) and $30 \log X \leq X^2.$ 
The last bound resembles \eqref{eq:6} except that it has an extra factor of $X\log X$ on the right--hand side. 
The last expression is $ M(X)^{(1+o(1))}$  when $X\rightarrow \infty$.  Notice that the logarithm of the right hand side of the last inequality  
is less than $10X$. So we have the desired result.\\

Assume now that $m!> \gamma  ^{n/2}$. Then $m\log m>\log m!>(n/2)\log \vert \gamma \vert$. So, one obtains 
$$ 
m> \dfrac{n\log \vert \gamma \vert }{2\log((n/2)\log \vert \gamma \vert)}.
$$ 
If 
$$
\dfrac{n\log \vert \gamma \vert }{2\log((n/2)\log \vert \gamma \vert)}<2X,
$$ 
then by Lemma \ref{3} with $s=1$, $x:=(n\log \vert \gamma \vert)/2$, and $T=2X$, we get
$$
n\log 2\le n\log \gamma<8X\log 2X,
$$
so $n<16X\log(2X)$, which is a very good bound on $n$. 

Assume now that   
$$
\dfrac{n\log \vert \gamma \vert }{2\log((n/2)\log \vert \alpha \vert)}\geq 2X.
$$ 
Thus, 
\begin{eqnarray*} 
\nu_p(m!)= \left\lfloor \frac{m}{p} \right\rfloor + \left\lfloor \frac{m}{p^2} \right\rfloor + \cdots
>\frac{m}{2p}\geq \dfrac{n\log \vert \gamma \vert }{4p\log((n/2)\log \vert \gamma \vert)}.
\end{eqnarray*}
If for some $p\leq X$ we have $\nu_p(m!)\leq \nu_p(u_n)$, then 
$$ \dfrac{n\log \vert \gamma \vert }{4p\log((n/2)\log \vert \gamma \vert)}\leq 1.2 \cdot 10^{12} \dfrac{p}{\log p}( \log p + \log Y) \log^2n,
$$ where we used Lemma \ref{lem:muun}. So, 
$$ 
n\leq 1.5\cdot 10^{13} \dfrac{p^2}{\log p}( \log p + \log Y) \log^3n,
$$ 
where we used the fact that $|\gamma|\ge 2$ and $\log((n/2)\log \vert \gamma \vert)\leq 2 \log n$ as  ($\log \vert \gamma \vert /2 < X< n$). 
So,
$$
n\le 3\cdot 10^{13} X^2(\log n)^3. 
$$
Applying Lemma \ref{3} with $s=3$, one get 
\begin{align*}
n &\leq 8\cdot 3 \cdot 10^{13}\cdot X^2 (\log (3 \cdot 10^{13}\cdot X^2))^3\\ 
&\leq 24\cdot 10^{13}\cdot X^2 (\log (3\cdot 10^{13})+ 2 \log X)^3\\
&\leq 8.1 \cdot 10^{16}\cdot X^2 \log^3X.
\end{align*}
The logarithm of the right hand side of the above relation is less than $9X$ for $X\ge 11$. So, we have the desired result.\\

If for all $p\leq X$, one has $\nu_p(m!)>\nu_p(u_n)$, then $\nu_p(Bs)=\nu_p(u_n)$ for all $p\leq X$ which implies  $\nu_p(s)\leq \nu_p(u_n)$. However, 
\begin{eqnarray}
\label{eq:s}
\log s &\leq  &\sum_{i=1}^{k} \nu_{p_i}(s) \log p_i\nonumber\\
&\leq & \pi(X) \nu_{p_i}(u_n)\log X\nonumber\\
&\leq &1.25 (1.2\cdot 10^{12})( 2X^2) \log^2n\nonumber\\
&\leq &3 \cdot 10^{12}X^2\log^2n,
\end{eqnarray}
where, for the third inequality, we used the fact that $\pi(X)\leq 1.25X / \log X$ and Lemma \ref{lem:muun}.
If $\alpha\ne \pm 1$, let $p$ be a prime dividing it. 
Notice that we have 
$$
\alpha p(n) = -(\Delta_1n+ \Delta_3)/(\alpha-\beta)^2.
$$ 
Since $\gcd (\alpha, \beta)=1$, it follows that $p \nmid (\alpha-\beta)^2$ and so $\nu_p(\alpha p(n))\geq 0$. 
We then have
\begin{eqnarray*}
\nu_p(-b\beta^n\pm Bs) &=& \mu_p(\alpha^np(n) - Am!)\\  
&\geq & \min \{\nu_p(m!), \nu_p(\alpha^{n-1})\}\\
&\geq  & \dfrac{n\log \vert \alpha \vert, }{4p\log((n/2)\log \vert \alpha \vert)}\\
& \ge & \dfrac{n}{16X\log n}.
\end{eqnarray*}
For the last inequality we used $\log |\alpha|\ge \log 2>1/2$ and $\log((n\log |\alpha|)/2)<2\log n$. If $\alpha=\pm 1$, we keep $\alpha^np(n)$ on the same side of the equation with $Bs$ 
and let $p$ be a prime factor of $\beta$. Write
$$
\alpha^np(n)\pm Bs=-b\beta^n+Am!,
$$
and a similar calculation gives us that 
$$
\nu_p(\alpha^n p(n)\pm Bs)>\frac{n}{16X\log n}.
$$
Thus, if $p\mid \alpha$ then 
$$ 
\nu_p(\beta^{n}b\mp Bs) = \nu_p(b)+ \nu_p(1\mp\beta^{-n} Bsb^{-1} ),\quad  \nu_p(b) \leq \log(4Y^3)/\log p<8\log X.
$$
while if $\alpha=\pm 1$, then $p\mid \beta$ and 
$$
\nu_p(\alpha^n p(n)\pm Bs)=\nu_p(p(n))+\nu_p(1\mp Bs p(n)^{-1}\alpha^{-n}),\quad  \nu_p(p(n))\le \frac{\log(4Y^3(n+1))}{\log p}<4\log n.
$$
Put $\Lambda = \vert 1\mp\gamma^{-n} BsC^{-1}\vert$, where $C=b$ if $\gamma=\beta$ and $C=p(n)$ if $\gamma=\alpha$. 
Clearly, $\Lambda \neq 0$ because we are working with a non-degenerate solution. We apply Yu's Theorem with the following parameters: 
$$
\eta_1=BC^{-1},~\eta_2=s,~\eta_3=\gamma,~d_1=1,~d_2=1,~d_3=-n.
$$
The heights of the involved numbers are bounded as follows:
\begin{itemize}
	\item[(i)] $h(\eta_1)\leq \log B + h(C^{-1}) \leq \log X+4\log n\le 5\log n$ (as $n+1< n^{1.5}$ and $4Y^3 < n^{2.5}$);
	\item[(ii)] $h(\eta_2)= \log s  \leq  3 \cdot 10^{12}\cdot X^2\cdot \log^2 n.$
	\item[(iii)] $h(\eta_3)= \log  \vert \gamma \vert \leq \log X.$  
\end{itemize}
By Lemma \ref{Y}, we have
\begin{eqnarray*}
\nu_p(\Lambda) &\leq & 19(20\cdot 2)^8\frac{p}{(\log p)^2} \log(3e^5)  (5\log n)(3\cdot 10^{12} X^2(\log n)^2) \log X \log n \\  
&\leq & 1.3\cdot 10^{28} X^3 \log X \log^4n.
\end{eqnarray*}
The previous computation implies 
$$
\frac{n}{16X\log n} \leq 1.5\cdot 10^{28} X^3 \log X \log^4n
$$
which give us 
$$
n\leq 24\cdot 10^{28} X^4 \log X \log^5n.
$$
Applying Lemma \ref{3} with $s=5$, one get 
\begin{align*}
n &\leq 2^5\cdot 24\cdot 10^{28} X^4 \log X (\log (24\cdot 10^{28} X^4 \log X))^5\\ 
&\leq 7.7\cdot 10^{30}\cdot X^4 \log X (\log (24\cdot 10^{28})+ 2 X)^5\\
&\leq 2.8 \cdot 10^{35}\cdot X^9 \log X,
\end{align*}
where for the second inequality, we used the fact that  $\log (X^4) < X$.
The logarithm of the right hand side of the above relation is less than $12X$ for $X\ge 11$. This finishes the proof of the theorem.

\section{Proof of Theorem \ref{thm:3}} \label{sec6}
We assume that $u_n=C_n=n2^n+1$ or $u_n=W_n=n2^n-1$. One has 
$$
f(X)=(X-2)^2(X-1)=X^3-5X^2+8X-4,
$$
and $C_0=1,~C_1=3,~C_2=9,~W_0=-1,~W_1=1,~W_2=7,~r_1=5,~r_2=-8,~r_3=4$ so in the equation 
$$
n2^n\pm 1=\pm m!\pm s
$$ 
with $s$ a positive integer whose prime factors are in $\{2,3,5,7\}$, we have $X=11$. Hence, by Theorem \ref{thm:2}, we have $n<e^{12\times 11}<10^{58}$. 

\begin{lem}\label{FA}
	There is no solution with $m\ge 500$.
\end{lem}

\begin{proof} 
Assume $m\ge 500$. Then $\nu_3(m!)\ge \nu_3(500!)= 247,~\nu_5(m!)\ge \nu_5(500!)=124,~\nu_7(m!)\ge \nu_7(500!)= 82$. The goal is to show that
$$
\nu_3(s)\le 124,\quad \nu_5(s)\le 98,\quad \nu_7(s)\le 79.
$$
Assume that $\nu_3(s)\ge 125$. Then $\nu_3(n2^n\pm 1)\ge 125$. We want to show that $n\geq 10^{58}$. The calculation is based on the following easy lemma.

\begin{lem}\label{Ho}
	For each fixed integer $t$ and  odd prime $p$ there are exactly $p-1$ numbers $n$ in $\{0,1,\ldots,p(p-1)-1\}$ such that $p\mid n2^n+1-t$.
\end{lem}

\begin{proof}
	This is implicit in work of Hooley \cite{Ho}. Let $a\in \{0,1,\ldots,p-2\}$. Then $a$ is a residue class modulo $p-1$. By Fermat's Little Theorem, if $n\equiv a\pmod {p-1}$, then $2^n\equiv 2^a\pmod p$. Hence, if $n2^n+1-t\equiv 0\pmod p$, then $n\equiv (t-1)2^{-a}\pmod p$. Thus, the residue class of $n$ modulo $p-1$ determines the residue class of $n$ modulo $p$ and such $n$ is uniquely determined modulo $p(p-1)$ by the Chinese Remainder Theorem. 
\end{proof}

Let $n_0\in \{0,\ldots,p(p-1)-1\}$ be such that $n\equiv n_0\pmod {p(p-1)}$ and $n2^n+1-t\equiv 0\pmod p$. We would like to find information about such $n$ knowing that $p^k\mid n2^n+1-t$. 
The argument is similar to Hensel's lemma. Here is the algorithm. Write 
$n=n_0+p(p-1)\ell$. We write the $p$-adic expansion of $\ell$, namely
$$
\ell=\ell_1+\ell_2 p+\cdots+\ell_k p^{k-1}+\cdots,\qquad \ell_i\in \{0,1,\ldots,p-1\}.
$$
Then
$$
n=n_0+(p-1)p\ell_1+(p-1)p^2 \ell_2+\cdots+(p-1)p^k\ell_k+\cdots.
$$
We find $\ell_i$ recursively in the following way assuming $p^{i+1}\mid n2^n+1-t$. Assume $j\ge 1$ and $\ell_1,\ldots,\ell_{j-1}$ has been determined and $p^{j}\mid n2^n+1-t$. Then setting
$$
n_{j-1}=n_0+(p-1)p\ell_1+\cdots+(p-1)p^{j-1}\ell_{j-1},
$$
we have $n_{j-1}2^{n_{j-1}}+1-t\equiv 0\pmod {p^j}$. Note that this is true for $j=1$. To find $\ell_j$, note that since $n\equiv n_{j-1}\pmod  {(p-1)p^j}$ and $(p-1)p^j=\phi(p^{j+1})$, it follows that
$$
2^{n}\equiv 2^{n_{j-1}}\pmod {p^{j+1}}.
$$
Thus, if $p^{j+1}\mid n2^n+1$, it then follows that 
\begin{eqnarray*}
	0 & \equiv & n2^n+1-t\pmod {p^{j+1}}\\
	& \equiv & n2^{n_{j-1}}+1-t \pmod {p^{j+1}}\\
	& \equiv & (n_{j-1}+(p-1)p^j \ell_j)2^{n_{j-1}}+1-t\pmod {p^{j+1}}\\
	& \equiv & (n_{j-1}2^{n_{j-1}}+1-t)+(p-1)p^j 2^{n_{j-1}}\ell_{j}\pmod {p^{j+1}}.
\end{eqnarray*}
Hence,
$$
\left(\frac{n_{j-1}2^{n_{j-1}}+1-t}{p^j}\right)+(p-1)2^{n_{j-1}}\ell_j\equiv 0\pmod p.
$$
Since $p-1\equiv -1\pmod p$, we get
$$
\ell_j\equiv 2^{-n_{j-1}}\left(\frac{n_{j-1}2^{n_{j-1}}+1-t}{p^j}\right)\pmod p.
$$
And we can continue.
To start, we make $t=0$ and put the above machine to work for $p=3$. In this case, $p(p-1)=6$ and there are two residue classes modulo $6$ such that $3\mid n2^n+1$, namely $n_0\in \{1,2\}$. When $n_0=1$ and $k=124$, the above process
(in Mathematica) gives
$$
n_{123}=14096601226371925780354191137048938941051110799238395669157
$$
while when $n_0=2$ and $k=124$, we get
$$
n_{123}=131916531426323976413079495561663150351720433293832571666642.
$$
In both cases $n_{125}>10^{58}$. This shows that in our range, $\nu_3(n2^n+1)<124$. Hence, $\nu_3(s)<124$. The same works for $p=5$ and $p=7$. We give the data:

For $p=5$, we take $k=99$. We have $p(p-1)=20$ and $n_0\in \{3,4,6,17\}$. We have 
\begin{eqnarray*}
	n_0=3, & &  n_{98}=3402055567449187211072479894744526992631911429806123056986882546322203;\\
	n_0=4, & & n_{98}=5860318539126309542028901497378642627938750361916774422262903402988764;\\
	n_0=6, & & n_{98}=6211271813369046855320209665842033651445457938030806323641242413003566;\\
	n_0=17, & & n_{98}=1900239201139363261324476300084028074211927656029119121314580491907717.\\
\end{eqnarray*}
In all cases $n_{100}>10^{69}>n$, so $\nu_5(n2^n+1)<99$. Hence, $\nu_5(s)<99$. For $p=7$ we take $k=79$. We have $(p-1)p=42$. We have $n_0\in \{5,6,10,26,27,31\}$. The data is
\begin{eqnarray*}
	n_0=5 & & n_{78}=23376667116957912273395168878053596583934978592913658754638298386469;\\
	n_0=6 & & n_{78}=26944746689754581236007271009151875823474002652201195796068635289134;\\
	n_0=10 & & n_{78}=24069582378334816208567848014057127858216459565384781083488608965992;\\
	n_0=26 & & n_{78}=6004003289610317916795511974189307812131311913908480006270103623040;\\
	n_0=27 & & n_{78}=9572082862406986879407614105287587051670335973196017047700440525705;\\
	n_0=31 & & n_{78}=6696918550987221851968191110192839086412792886379602335120414202563.
\end{eqnarray*}
In all cases $n_{80}>10^{66}>n$, so $\nu_7(n2^n+1)<79$, so $\nu_7(s)<79$.

When $t=2$, we get information about the exponents of the small primes in $W_n$. Let $p=3$. In this case, $p(p-1)=6$ and there are two residue classes modulo $6$ such that $3\mid n2^n-1$ according to Lemma \ref{FA}, namely $n_0\in \{4,5\}$. When $n_0=4$ and $k=126$, the above process
(in Mathematica) gives
$$
n_{125}=1324117109863992278171562286849551012905296843274331852235486
$$
while when $n_0=5$ and $k=126$, we get
$$
n_{125}=2024168377236220040978157856035277257188964269091189786706895.
$$
In both cases, $n_{125}>10^{58}$. This shows that in our range, $\nu_3(n2^n-1)<126$. Hence, $\nu_3(s)<125$. The same works for $p=5$ and $p=7$. We give the data:

For $p=5$, we take $k=99$. We have $p(p-1)=20$ and $n_0\in \{7,13,14,16\}$. We have 
\begin{eqnarray*}
	n_0=7, & &  n_{98}=50556828220234104829713905612151729929047533964652111403956962145639
	67;\\
	n_0=13, & & n_{98}=246611946565139989425565633613382073939085689370031037905766823665953;\\
	n_0=14, & & n_{98}=27048749182422623203819872362474977092459246214806824031817876803325
	14;\\
	n_0=16, & & n_{98}=30558281924849996336732954047108887327526321975947143045601266903473
	16.\\
\end{eqnarray*}
In all cases $n_{98}>10^{58}>n$, so $\nu_5(n2^n-1)<99$. Hence, $\nu_5(s)<99$. For $p=7$ we take $k=79$. We have $(p-1)p=42$. We have $n_0\in \{2,4,15,23,25,36\}$. The data is
\begin{eqnarray*}
	n_0=2 & & n_{78}=3070945089242253569511128531703 7993482895779452876147449227324975278;\\
	n_0=4 & & n_{78}=4084723421753861202636449976
	25665865881992651094052930276211831026;\\
	n_0=15 & & n_{78}=2157106343186897994855204068753 0035249121277125785855969465402463679;\\
	n_0=23 & & n_{78}=13336787065074941338511628413173 704711092112773870968700859130211849;\\
	n_0=25 & & n_{78}=177811361695229804768633019014899 54637685659330099231678644406594455;\\
	n_0=36 & & n_{78}=4198399604521385591952383783665746 477317610446780677221097207700250.
\end{eqnarray*}
In all cases $n_{80}>10^{58}>n$, so $\nu_7(n2^n-1)<79$, so $\nu_7(s)<79$.

 Now we calculate 
$$
\max\,^{*}\{\nu_2(3^a\cdot 5^b\cdot 7^c\pm 1): 0\le a\le 125,~0\le b\le 99,~0\le c\le 79\}
$$
where $^*$ means that we calculate the maximum only over the triples $(a,b,c)$ such that the number we apply $\nu_2$ to is nonzero (that is, we exclude the case of the negative sign when $a=b=c=0$). We obtain that the above maximum is 
at most $19$. Since $\nu_2(m!)\ge \nu_2(500!)=494$, we get $n\le 19$. Thus, 
$$
100^{500}<(m/e)^m<m!=|s\pm (n2^n\pm 1)|<3^{125}\cdot 5^{99}\cdot 7^{80}+19\cdot 2^{19}+1<10^{200},
$$
a contradiction. This shows that $m\le 500$. 

\end{proof}

Lemma \ref{Ho} gives us much more than just that $m\le 500$. It also suggests how we should go about finishing the proof. Namely, we take $m\in [2,500]$ and fix the sign $\varepsilon\in \{\pm 1\}$  and calculate the largest possible power of $p$ in $C_n+\varepsilon m!$ for $p\in \{3,5,7\}$ with a similar procedure. 
Namely, we first loop over all possible $n_0$ to find the $p-1$ values in $\{0,1,\ldots,(p-1)p-1\}$ such that if $n\equiv n_0\pmod {(p-1)p}$, then $n2^n\pm 1-\varepsilon m!\equiv 0\pmod p$. 
Then we get : 
\begin{itemize}
\item Case 1 $(C_n-m!)$, we have $\nu_3(s)<126,~\nu_5(s)<100,~\nu_7(s)<80$.
\item Case 2 $(W_n+m!)$, we have $\nu_3(s)<127,~\nu_5(s)<100,~\nu_7(s)<80$.
\item Case 3 $(C_n+m!)$, we have $\nu_3(s)<129,~\nu_5(s)<100,~\nu_7(s)<80$.
\item Case 4 $(W_n-m!)$, we have $\nu_3(s)<129,~\nu_5(s)<100,~\nu_7(s)<80$.
\end{itemize}
Indeed we apply the above algorithm to compute 
$n_k$ for $(p,k)=(3,126),~(5,100),\\ ~(7,80)$. In all the four cases above we get that $n_k>10^{58}$ for all choices of $m\in [1,500]$, provided the lower bounds 
on $\nu_p(s)$ exceed the numbers indicated above.	Now ran another loop over 
$m\in [2,500]$, $a\in [0,130],~b\in [0,100],~c\in [0,80]$ and showed that $\nu_2(3^a\cdot 5^b\cdot 7^c\pm 1\pm m!)<30$, whenever the number inside  $\nu_2$ is nonzero.  This shows that $n\le 30$. Then we generated all numbers of the form $C_n\pm m!$ and $W_n\pm m!$ for $n\in [0,30],~m\in [2,500]$ and intersected this set with the set of numbers $\{\pm 3^a\cdot 5^b\cdot 7^c\}$ where $0\le a\le 130,~0\le b\le 100,~0\le c\le 80$. This intersection is 
$$
\{-25,-21,-7,-5-3,-1,3,5,7,9,15,21,25,27,49,63,135,175,729,2025,5103\}.
$$
The corresponding solutions are 
\begin{eqnarray*}
1& = &  W_0+2! =  C_1-2!^\dag  =  W_2-3!  =  C_3-4!^\dag;\,~ -1  =  C_0-2! =  W_1-2! ^\dag =  W_3-4!^\dag;\\
3  & = & C_0+2!  =  W_1+2!  =  C_2-3!;  \qquad\qquad\quad\,~\,~\,~ -3 =  W_0-2!  =  C_1-3!;  \\
5 & = & W_0+3!  =  C_1+2! =  W_2-2! ;\qquad\qquad\quad\,~\,~ -5  =  C_0-3!  =  W_1-3!;   \\
7 & = & C_0+3!  =  W_1+3!  =  C_2-2! ;\qquad \qquad\quad\,~\,~\,~ -7  =  W_0-3!; \\
3^2 &  = & C_1+3! =  W_2+2!; \\
3\cdot 5 & = & C_2+3!;  \\
3\cdot 7 & = & W_3-2!; \qquad\qquad\qquad\qquad\qquad\qquad\,~\,~\,~\,~\,~\,~\,~\,~ -3\dot 7 =  C_1-4!;   \\
5^2 & = & W_3+2!;  \qquad\qquad\qquad\qquad\qquad\qquad\,~\,~\,~\,~\,~\,~\,~\,~  -5^2   =  W_1-4!; \\
3^3 & = & C_3+2!;  \\
7^2 & = & C_3+4!;  \\
3^2\cdot 7 & = & C_4-2!;  \\
3^3\cdot 5 & = & W_5-4!;  \\
5^2\cdot 7 & = & W_7-5!;  \\
3^6 & = & C_2+6!;  \\
3^4\cdot 5^2 & = & C_8-4!;  \\
3^5\cdot 7 & = & W_4+7!; 
\end{eqnarray*}
The solutions indicated with ${}^\dag$ are degenerate.

\bibliographystyle{plain}

\end{document}